\documentclass[a4,11pt]{amsart}

\usepackage[T1]{fontenc}
\usepackage[utf8]{inputenc}
\usepackage{lmodern}
\usepackage[english]{babel}
\usepackage[autostyle]{csquotes}

\usepackage{filecontents}
\usepackage{enumitem}

\usepackage{amsmath,amssymb,amsthm,amscd,amsaddr}
\newtheorem{thm}{Theorem}[section]
\newtheorem{prop}[thm]{Proposition}
\newtheorem{lem}[thm]{Lemma}

\newtheorem{ex}[thm]{Example}
\newtheorem{defn}[thm]{Definition}
\theoremstyle{remark}
\newtheorem{rem}[thm]{Remark}

\begin{document}
\title{Property A and coarse embeddability for fuzzy metric spaces}
\author{Yeong Chyuan Chung}
\address{Mathematisch Instituut, Universiteit Leiden, Niels Bohrweg 1, 2333 CA Leiden, The Netherlands}
\email{y.c.chung@math.leidenuniv.nl}
\date{\today}
\thanks{This research did not receive any specific grant from funding agencies in the public, commercial, or not-for-profit sectors.}
\maketitle

\begin{abstract}
Property A is a geometric property originally introduced for discrete metric spaces to provide a sufficient condition for coarse embeddability into Hilbert space, and it is defined via a F\o{}lner condition similar in spirit to the classical notion of amenability for groups.
In this paper, we define property A for fuzzy metric spaces in the sense of George and Veeramani, show that it is an invariant in the coarse category of fuzzy metric spaces, and provide characterizations of it for uniformly locally finite fuzzy metric spaces.
We also show that uniformly locally finite fuzzy metric spaces with property A are coarsely embeddable into Hilbert space.
\end{abstract}

%\tableofcontents

\section{Introduction}

A form of probabilistic geometry was introduced by Menger \cite{M1,M2,M3,M4} who axiomatized the idea that the distance between two points is probabilistic rather than deterministic by replacing the numerical-valued distance between two points by a cumulative distribution function, reflecting in a way the uncertainty inherent in measurements. In fact, his definition of an ``ensemble flou'' in \cite{M4} is identical to the definition of a fuzzy set later given by Zadeh \cite{Z}, thereby anticipating the theory of fuzzy sets by more than a decade.

Menger's idea was picked up and developed by others, notably Schweizer and Sklar \cite{SS1,SS2}.
Motivated by these earlier works, Kramosil and Mich\'{a}lek introduced fuzzy metric spaces in \cite{KM}. George and Veeramani gave a modified definition in \cite{GV} such that the induced topology is Hausdorff, and this is the definition we will use in this paper.
Besides theoretical interest in the subject, such as topological and fixed point properties, fuzzy metrics have found applications in engineering problems, notably in image processing filters that produce better quality and sharpness than classical ones built using metrics \cite{GMS,LL,RKP}.

Coarse geometry can be briefly described as the study of geometric objects viewed from afar. Coarse geometric ideas were already present in Mostow's rigidity theorem \cite{Mostow} and work on growth of groups \cite{Gromov}. The subject of coarse geometry has also formed connections with Banach space theory \cite{Ost}, topology and index theory \cite{Roe,Yu98}, as well as operator algebras and noncommutative geometry \cite{Yu95,Yu00}. Most of these involve the coarse geometry of metric spaces.

Zarichnyi wrote a short note \cite{Zar} on the coarse geometry of fuzzy metric spaces, and Grzegrzolka studied the asymptotic dimension of fuzzy metric spaces in \cite{Grz}. In this paper, we take the coarse geometric study of fuzzy metric spaces a step further by studying property A and coarse embeddability for uniformly locally finite fuzzy metric spaces.
Property A is a geometric property that was introduced for discrete metric spaces in \cite{Yu00}. It is similar in spirit to, and inspired by, the classical notion of amenability for groups. Property A for a space $X$ is defined via a F\o{}lner condition on $X\times\mathbb{N}$, and is an invariant for the coarse geometry of the space. Moreover, just as there are numerous equivalent conditions for amenability, property A also comes in many guises as shown in \cite[Theorem 1.11]{Willett} and \cite[Theorem 5.1]{Sako14}.

The motivation for the introduction of Property A in \cite{Yu00} was to provide a sufficient condition for coarse embeddability into Hilbert space. It is also known that discrete, uniformly locally finite metric spaces with finite asymptotic dimension have property A \cite[Lemma 4.3]{HR}.
We show that both remain true for uniformly locally finite fuzzy metric spaces.
As a consequence, one can use Hilbert space techniques to study the large scale structure of uniformly locally finite fuzzy metric spaces with property A.

We note that Sako considered property A for uniformly locally finite coarse spaces in \cite[Definition 2.7]{Sako} using a definition equivalent to the one given by Roe in \cite[Definition 11.35]{RoeL}. The definition still makes sense without uniform local finiteness, and was considered in the doctoral thesis of Bunn \cite[Definition 5.5]{Bunn}. The definition we give for fuzzy metric spaces agrees with Sako's if one considers the coarse structure given by sets that are bounded with respect to the fuzzy metric.
It was shown in \cite[Chapter 11]{RoeL} that for uniformly locally finite coarse spaces, finite asymptotic dimension implies property A, which in turn implies coarse embeddability into Hilbert space.
Thus, the results in this paper may be seen as a special case when applied to fuzzy metric spaces with their bounded coarse structure.
Nevertheless, we have in mind an audience interested in fuzzy metric spaces and not necessarily having knowledge about coarse spaces, so the main body of this paper is written entirely in the language of fuzzy metric spaces, and we have an appendix explaining how our results can be recovered from the more general ones about coarse spaces.

We end this introduction with an outline of the rest of the paper.
In section 2, we recall basic definitions and properties pertaining to fuzzy metric spaces in the sense of George and Veeramani.
In section 3, we define property A for uniformly locally finite fuzzy metric spaces, and show that fuzzy metric spaces with finite asymptotic dimension have property A.
In section 4, we show that property A is preserved by coarse equivalences, and is in fact inherited by fuzzy metric subspaces.
In section 5, we provide characterizations of property A for uniformly locally finite fuzzy metric spaces.
In section 6, we show that uniformly locally finite fuzzy metric spaces with property A coarsely embed into Hilbert space.
Finally, we have an appendix about coarse spaces and how our results can be obtained from results about coarse spaces.

\section{Fuzzy metric spaces}

In this section, we record basic definitions and properties pertaining to fuzzy metric spaces as defined in \cite{GV}.

\begin{defn}
A binary operation \[ *:[0,1]\times[0,1]\rightarrow[0,1] \] is called a continuous $t$-norm if
\begin{enumerate}
\item $*$ is associative and commutative,
\item $*$ is continuous,
\item $a*1=a$ for all $a\in[0,1]$,
\item $a*b\leq c*d$ whenever $a\leq c$ and $b\leq d$ for $a,b,c,d\in[0,1]$.
\end{enumerate}
We will write $a^{*(m)}$ for $a*a*\cdots *a$ ($m$ times).
\end{defn}

\begin{defn}
A triple $(X,M,*)$ is a fuzzy metric space if $X$ is a set, $*$ is a continuous $t$-norm, and $M:X\times X\times(0,\infty)\rightarrow[0,1]$ is a function satisfying the following conditions for all $x,y,z\in X$ and all $s,t>0$:
\begin{enumerate}
\item $M(x,y,t)>0$,
\item $M(x,y,t)=1$ if and only if $x=y$,
\item $M(x,y,t)=M(y,x,t)$,
\item $M(x,y,t)*M(y,z,s)\leq M(x,z,t+s)$,
\item $M(x,y,\cdot):(0,\infty)\rightarrow[0,1]$ is continuous.
\end{enumerate}
\end{defn}

It was shown in \cite[Lemma 4]{Grab} that $M(x,y,\cdot)$ is non-decreasing for $x,y \in X$.

\begin{ex}
Let $(X,d)$ be a metric space, and let $a*b=ab$ for $a,b\in[0,1]$. Define \[ M(x,y,t)=\frac{t}{t+d(x,y)}. \]
Then $(X,M,*)$ is a fuzzy metric space, which we call the standard fuzzy metric space induced by the metric $d$. 
\end{ex}

\begin{ex} \cite[Example 2.11]{GV}
Let $X=\mathbb{N}$, define $a*b=ab$, and define
\[ M(x,y,t)=\begin{cases} x/y &\;\text{if}\; x\leq y \\ y/x &\;\text{if}\; y\leq x \end{cases} \] for all $t>0$.
Then $(X,M,*)$ is a fuzzy metric space. Moreover, $M$ is not the standard fuzzy metric corresponding to any metric on $X$.
\end{ex}

\begin{defn}
Let $(X,M,*)$ be a fuzzy metric space. The ball $B(x,r,t)$ with center $x\in X$, $r\in(0,1)$, and $t>0$ is defined by
\[ B(x,r,t)=\{y\in X:M(x,y,t)>1-r\}. \]
\end{defn}

If $(X,d)$ is a metric space, then we may compare balls in $(X,d)$ with balls in the corresponding standard fuzzy metric space.

\begin{lem} \label{lem:balls}
Let $(X,d)$ be a metric space, and let $(X,M,*)$ be the corresponding standard fuzzy metric space. For $R>0$, let $B_d(x,R)=\{y\in X:d(x,y)<R\}$. Then 
\[ B_d(x,R)= B\left(x,\frac{R}{t+R},t\right)=B\left(x,r,\frac{R(1-r)}{r}\right) \] for each $x\in X$, $t>0$, $r\in(0,1)$, and $R>0$.
In particular, $B_d(x,R)=B(x,\frac{1}{2},R)$ for each $x\in X$ and $R>0$.
\end{lem}

\begin{proof}
The proof is a straightforward computation using the definition of the standard fuzzy metric space.
\end{proof}

\begin{defn}
Let $(X,M,*)$ be a fuzzy metric space. A subset $A$ of $X$ is bounded if there exist $r\in(0,1)$ and $t>0$ such that \[ M(x,y,t)>1-r \] for all $x,y\in A$.
\end{defn}

Throughout this paper, we will follow \cite{Grz} in assuming that $a*b\neq 0$ whenever $a\neq 0$ and $b\neq 0$, i.e., there are no zero divisors with respect to $*$. This ensures that the union of two bounded sets is bounded \cite[Proposition 2.8]{Grz}.
For example, the \L ukasiewicz $t$-norm given by $a*b=\max(a+b-1,0)$ will be excluded by this assumption.
We refer the reader to \cite[Example 2.7]{Grz} for an example of a fuzzy metric space in which the union of two bounded sets is unbounded.

With this assumption, the collection of all subsets $E$ of $X\times X$ for which there exist $r\in(0,1)$ and $t>0$ such that $M(x,y,t)>1-r$ for all $(x,y)\in E$ forms a coarse structure on $X$ in the sense of \cite[Definition 2.3]{RoeL}.

\section{Property A for fuzzy metric spaces}

In this section, we define property A for uniformly locally finite fuzzy metric spaces, and show that uniformly locally finite fuzzy metric spaces with finite asymptotic dimension have property A.
Uniform local finiteness entails a certain degree of discreteness, which we make more precise below.

\begin{defn} \label{ulfFM}
A fuzzy metric space $(X,M,*)$ is said to be uniformly locally finite if for every $r\in(0,1)$ and $t>0$, there exists $N_{r,t}>0$ such that $|B(x,r,t)|<N_{r,t}$ for all $x\in X$.
\end{defn}

Note that if a fuzzy metric space $(X,M,*)$ is uniformly locally finite, then $X$ is necessarily countable.
%Indeed, given $x\in X$, if $y\neq x$ and $n\in\mathbb{N}$, then $M(x,y,n)=1-r$ for some $r\in(0,1)$.
%Choosing $m\in\mathbb{N}$ such that $\frac{1}{m}>r\geq\frac{1}{m+1}$, we get $M(x,y,n)>1-\frac{1}{m}$, i.e., $y\in B(x,\frac{1}{m},n)$.
%Therefore, for a fixed $x\in X$, we have $X=\bigcup_{n,m\in\mathbb{N}}B(x,\frac{1}{m},n)$, so $X$ is countable since each ball is finite. 
%Also, for each $x\in X$, $t>0$, and $m\in\mathbb{N}$, there exists $R\in(0,\frac{1}{m})$ such that $B(x,R,t)=\{x\}$.
%Indeed, let $B(x,\frac{1}{m},t)=\{x,x_1,\ldots,x_n\}$, and let $R=1-\max_iM(x,x_i,t)\in(0,\frac{1}{m})$.
%Then $M(x,x_i,t)\leq 1-R$ for $i=1,\ldots,n$.
%If $y\in X\setminus B(x,\frac{1}{m},t)$, then $M(x,y,t)\leq 1-\frac{1}{m}<1-R$. Hence, $B(x,R,t)=\{x\}$.

\begin{defn}
A uniformly locally finite fuzzy metric space $(X,M,*)$ is said to have property A if for every $\varepsilon>0$, $t>0$, and $r\in(0,1)$, there exist $t'>0$, $r'\in(0,1)$, and a family $\{A_x\}_{x\in X}$ of non-empty finite subsets of $X\times\mathbb{N}$ such that
\begin{enumerate}
\item $A_x\subset B(x,r',t')\times\mathbb{N}$ for each $x\in X$,
\item the symmetric difference $A_x\Delta A_y$ satisfes $|A_x\Delta A_y|<\varepsilon|A_x\cap A_y|$ when $M(x,y,t)>1-r$.
\end{enumerate}
\end{defn}

This definition is inspired by the original definition of property A for metric spaces in \cite{Yu00}, which we now recall.

\begin{defn}
A discrete metric space $(X,d)$ is said to have property A if for every $\varepsilon>0$ and $R>0$, there exist $S>0$ and a family $\{A_x\}_{x\in X}$ of non-empty finite subsets of $X\times\mathbb{N}$ such that
\begin{enumerate}
\item $A_x\subset B_d(x,S)\times\mathbb{N}$ for each $x\in X$,
\item the symmetric difference $A_x\Delta A_y$ satisfes $|A_x\Delta A_y|<\varepsilon|A_x\cap A_y|$ when $d(x,y)<R$.
\end{enumerate}
\end{defn}

The original definition in \cite{Yu00} had one more condition, namely that $(x,1)\in A_x$ for each $x\in X$, but we have followed subsequent authors, such as \cite[Definition 1.1]{Willett}, in dropping the condition.

One will notice that our definition of property A still makes sense without the condition of uniform local finiteness, but in Theorem \ref{charThm} one will see that this condition is used to get certain characterizations of property A.
Moreover, this condition is often assumed in the context of metric spaces, and is also assumed in the definition of property A for coarse spaces in \cite{Sako}.

\begin{prop}
Let $(X,d)$ be a uniformly locally finite metric space, and let $(X,M,*)$ be the corresponding standard fuzzy metric space. Then $(X,d)$ has property A if and only if $(X,M,*)$ has property A.
\end{prop}

\begin{proof}
We will use Lemma \ref{lem:balls} to compare balls in $(X,d)$ with balls in $(X,M,*)$.

Suppose $(X,d)$ has property A.
Given $\varepsilon>0$, $t>0$, and $r\in(0,1)$, let $R=\frac{tr}{1-r}$.  
There exist $S>0$ and a family $\{A_x\}_{x\in X}$ of non-empty finite subsets of $X\times\mathbb{N}$ such that $A_x\subset B_d(x,S)\times\mathbb{N}\subseteq B(x,\frac{S}{t+S},t)\times\mathbb{N}$ for each $x\in X$, and $|A_x\Delta A_y|<\varepsilon|A_x\cap A_y|$ whenever $y\in B_d(x,R)$ and thus whenever $y\in B(x,r,t)$.
Hence, $(X,M,*)$ has property A.

Conversely, suppose $(X,M,*)$ has property A.
Given $R>0$ and $\varepsilon>0$, let $t=R$ and $r=\frac{1}{2}$.
There exist $t'>0$, $r'\in(0,1)$, and a family $\{A_x\}_{x\in X}$ of non-empty finite subsets of $X\times\mathbb{N}$ such that $A_x\subset B(x,r',t')\times\mathbb{N}\subseteq B_d(x,\frac{t'r'}{1-r'})\times\mathbb{N}$ for each $x\in X$, and $|A_x\Delta A_y|<\varepsilon|A_x\cap A_y|$ whenever $M(x,y,t)>1-r=\frac{1}{2}$ and thus whenever $d(x,y)<R$.
\end{proof}

Our definition is compatible with Sako's definition of property A for general coarse spaces in \cite[Definition 2.7]{Sako} if one considers the coarse structure given by the collection of all subsets $E$ of $X\times X$ for which there exist $r\in(0,1)$ and $t>0$ such that $M(x,y,t)>1-r$ for all $(x,y)\in E$. We refer the reader to Proposition \ref{Aequiv} for a proof of this statement.

We now proceed to show that fuzzy metric spaces with finite asymptotic dimension have property A, and we begin by recalling the definition of asymptotic dimension for fuzzy metric spaces from \cite{Grz}.

\begin{defn}
A family $\mathcal{U}$ of subsets of a fuzzy metric space $(X,M,*)$ is uniformly bounded if there exist $r\in(0,1)$ and $t>0$ such that $M(x,y,t)>1-r$ for all $x,y\in U$ and $U\in\mathcal{U}$.
\end{defn}

\begin{defn}
Two subsets $U$ and $U'$ of a fuzzy metric space $(X,M,*)$ are $(r,t)$-disjoint for some $r\in(0,1)$ and $t>0$ if $M(U,U',t)<1-r$, where $M(U,U',t)=\sup\{M(x,y,t):x\in U,y\in U'\}$.

A family $\mathcal{U}$ of subsets of $(X,M,*)$ is said to be $(r,t)$-disjoint if any two distinct members of $\mathcal{U}$ are $(r,t)$-disjoint.
\end{defn}

\begin{defn}
Let $(X,M,*)$ be a fuzzy metric space. We say that the asymptotic dimension of $X$ does not exceed $n$, and write $asdim(X)\leq n$, if for every $r\in(0,1)$ and $t>0$, there exist $(r,t)$-disjoint families $\mathcal{U}^0,\ldots,\mathcal{U}^n$ of subsets of $X$ such that $\bigcup_i\mathcal{U}^i$ is a uniformly bounded cover of $(X,M,*)$. If no such $n$ exists, we say that $asdim(X)=\infty$.
If $asdim(X)\leq n$ but $asdim(X)\nleq n-1$, then we say that $asdim(X)=n$.
\end{defn}

\begin{defn}
A Lebesgue pair of a cover $\mathcal{U}$ of a fuzzy metric space $(X,M,*)$ is a pair of numbers $r\in(0,1)$ and $t>0$ such that for any $x\in X$, $B(x,r,t)\subseteq U$ for some $U\in\mathcal{U}$.

The multiplicity of $\mathcal{U}$ is the smallest integer $n$ such that every point $x\in X$ is contained in at most $n$ members of $\mathcal{U}$.
\end{defn}

\begin{lem} \cite[Theorem 4.5]{Grz}
Let $(X,M,*)$ be a fuzzy metric space. If $asdim(X)\leq n$, then for every $r\in(0,1)$ and $t>0$, there exists a uniformly bounded cover $\mathcal{U}$ of $X$ with Lebesgue pair $(r,t)$ and multiplicity at most $n+1$.
\end{lem}

The main result of this section is analogous to \cite[Theorem 2.1]{CDV}, and shows that fuzzy metric spaces with finite asymptotic dimension have property A.

\begin{thm} \label{thm:asdim}
Let $(X,M,*)$ be a uniformly locally finite fuzzy metric space, and $n\geq 0$. 
If $asdim(X)\leq n$, then for each $\varepsilon>0$, $t>0$, and $r\in(0,1)$, there exist $t'>0$, $r'\in(0,1)$, and non-empty finite subsets $A_x\subset B(x,r',t')\times\mathbb{N}$ for $x\in X$ such that $|A_x\Delta A_y|<\varepsilon|A_x\cap A_y|$ when $M(x,y,t)>1-r$, and the projection of $A_x$ onto $X$ contains at most $n+1$ elements for each $x\in X$.
%The following are equivalent: \textbf{(iii) implies (i) may be a problem; need converse of lemma}
%\begin{enumerate}
%\item $asdim(X)\leq n$.
%\item For each $\varepsilon>0$, $t>0$, and $r\in(0,1)$, there exist $t'>0$, $r'\in(0,1)$, and non-empty finite subsets $A_x\subset B(x,r',t')\times\mathbb{N}$ for $x\in X$ such that $|A_x\Delta A_y|<\varepsilon|A_x\cap A_y|$ when $M(x,y,t)>1-r$, and the projection of $A_x$ onto $X$ contains at most $n+1$ elements for each $x\in X$.
%\item For each $t>0$, and $r\in(0,1)$, there exist $t'>0$, $r'\in(0,1)$, and non-empty finite subsets $A_x\subset B(x,r',t')\times\mathbb{N}$ for $x\in X$ such that $|A_x\Delta A_y|<|A_x\cap A_y|/(n+1)$ when $M(x,y,t)>1-r$, and the projection of $A_x$ onto $X$ contains at most $n+1$ elements for each $x\in X$.
%\end{enumerate}
In particular, fuzzy metric spaces with finite asymptotic dimension have property A.
\end{thm}

\begin{proof}
%(i) $\Rightarrow$ (ii):
Assume $asdim(X)\leq n$, and fix $\varepsilon>0$, $t>0$, and $r\in(0,1)$.
Let $\mathcal{U}$ be a uniformly bounded cover of $X$ with multiplicity at most $n+1$ and with Lebesgue pair $(R,T)$, where $R=1-(1-r)^{*(\lfloor 2+(2n+1)/\varepsilon \rfloor+1)}\in(0,1)$ and $T=(2+\frac{2n+1}{\varepsilon})t$.
Here, $\lfloor 2+\frac{2n+1}{\varepsilon} \rfloor$ denotes the integer part of $2+\frac{2n+1}{\varepsilon}$.
Note that $1-R\leq 1-r$ and $t<T$, so $B(x,r,t)\subseteq B(x,R,T)$ for each $x\in X$.

Since $\mathcal{U}$ is uniformly bounded, there exist $r'\in(0,1)$ and $t'>0$ such that $M(x,y,t')>1-r'$ for all $x,y\in U$ and $U\in\mathcal{U}$. Pick an element $a_U\in U$ for each $U\in\mathcal{U}$.
We call a finite sequence $x_0,\ldots,x_n$ of points in $X$ an $(r,t)$-chain from $x_0$ to $x_n$ if $M(x_i,x_{i-1},t)>1-r$ for $i=1,\ldots,n$.
For $x\in X$ and $U\in\mathcal{U}$, let $l_U(x)$ denote the length of the shortest $(r,t)$-chain from $x$ to a point outside $U$. If there is no such chain, set $l_U(x)$ equal to $\lfloor 2+\frac{2n+1}{\varepsilon} \rfloor$.
Now let \[ A_x=\bigcup_{U\ni x}\{a_U\}\times\{1,\ldots,l_U(x)\}. \]
This is a finite set for each $x\in X$ since the cover $\mathcal{U}$ has finite multiplicity.
If $x\in U$, then $M(x,a_U,t')>1-r'$, so $A_x\subset B(x,r',t')\times\mathbb{N}$ for each $x\in X$.
Since $\mathcal{U}$ has multiplicity at most $n+1$, the projection of $A_x$ onto $X$ contains at most $n+1$ elements for each $x\in X$.

Assume $M(x,y,t)>1-r$ for the rest of the proof. 
If $y,y_1,\ldots,y_k$ forms an $(r,t)$-chain, then so does $x,y,y_1,\ldots,y_k$.
If there is no $(r,t)$-chain from $x$ to $X\setminus U$, then there is also no $(r,t)$-chain from $y$ to $X\setminus U$. 
Thus, $|l_U(x)-l_U(y)|\leq 1$.
Since $\mathcal{U}$ has Lebesgue pair $(R,T)$ and multiplicity at most $n+1$, there are at most $2n+1$ elements of $\mathcal{U}$ containing either $x$ or $y$.
Moreover, if $x\in U$ and $y\notin U$, then $l_U(x)=1$ since $x,y$ form an $(r,t)$-chain. 
Thus, \[|A_x\Delta A_y|\leq 2n+1.\]

By definition of a Lebesgue pair, there exists $U_0\in\mathcal{U}$ such that $B(x,R,T)\subset U_0$.
If $x,x_1,\ldots,x_m$ forms an $(r,t)$-chain, and $x_m\in X\setminus U_0$, then $M(x,x_m,T)\leq 1-R$. Now if $m<\lfloor 2+\frac{2n+1}{\varepsilon} \rfloor$, then
\begin{align*}
1-R\geq M(x,x_m,T) &= M(x,x_m,(2+\frac{2n+1}{\varepsilon})t) \\
&\geq M(x,x_m,mt) \\ 
&\geq M(x,x_1,t)*M(x_1,x_2,t)*\cdots*M(x_{m-1},x_m,t) \\
&\geq (1-r)^{*(m)}.
\end{align*}
Since $1-R=(1-r)^{*(\lfloor 2+(2n+1)/\varepsilon \rfloor+1)}$, we get $\lfloor 2+\frac{2n+1}{\varepsilon} \rfloor+1\leq m<\lfloor 2+\frac{2n+1}{\varepsilon} \rfloor$, which is a contradiction.
Hence, $m\geq\lfloor 2+\frac{2n+1}{\varepsilon} \rfloor$.

Thus, $\{a_{U_0}\}\times\{1,\ldots,\lfloor 2+\frac{2n+1}{\varepsilon} \rfloor-1\}\subseteq A_x\cap A_y$, and 
\[|A_x\cap A_y|\geq \lfloor 2+\frac{2n+1}{\varepsilon} \rfloor-1>\frac{2n+1}{\varepsilon}.\]
Consequently, \[ \frac{|A_x\Delta A_y|}{|A_x\cap A_y|}<\frac{2n+1}{(2n+1)/\varepsilon}=\varepsilon. \]

%(ii) $\Rightarrow$ (iii) is immediate by setting $\varepsilon=\frac{1}{n+1}$.
\end{proof}

The fuzzy metric space in the next example is not the standard fuzzy metric space corresponding to any metric, and it was shown to have asymptotic dimension zero in \cite[Example 6.3]{Grz}. It has property A by the theorem above but we can also show this directly using the definition.

\begin{ex}
Let $X=\mathbb{N}$, define $a*b=ab$, and define
\[ M(x,y,t)=\begin{cases} 1 &\;\text{if}\; x= y \\ 1/xy &\;\text{if}\; x\neq y \end{cases} \] for all $t>0$.
Given $\varepsilon>0$, $r\in(0,1)$, and $t>0$, choose $N\in\mathbb{N}$ such that $\frac{1}{N+1}<1-r$.
For $n\in\mathbb{N}$, set \[ A_n=\begin{cases} \{N\}\times\{1\} &\;\text{if}\; n\leq N \\ \{n\}\times\{1\} &\;\text{if}\; n>N \end{cases} \]
For $n<N$, we have $M(n,N,t)=\frac{1}{nN}>\frac{1}{N^2}$, and it follows that $A_n\subset B(n,1-\frac{1}{N^2},t)\times\mathbb{N}$ for all $n\in\mathbb{N}$.
If $M(n,m,t)>1-r$ and $m\neq n$, then $\frac{1}{mn}>1-r>\frac{1}{N+1}$, so $A_n=A_m=\{N\}\times\{1\}$, and it follows that $|A_n\Delta A_m|<\varepsilon|A_n\cap A_m|$ whenever $M(n,m,t)>1-r$.
Hence, $(X,M,*)$ has property A.
\end{ex}

Dranishnikov \cite[Section 4]{Dra} gave examples of metric spaces with both infinite asymptotic dimension and property A, so the standard fuzzy metric spaces corresponding to these metric spaces also have both properties.

\section{Coarse invariance}

In this section, we show that property A is preserved by coarse equivalences. In fact, it is a property that is inherited by fuzzy metric subspaces.

\begin{defn} \label{ceFM}
Let $(X,M_1,*_1)$ and $(Y,M_2,*_2)$ be fuzzy metric spaces. Let $f:X\rightarrow Y$ be a map.
\begin{enumerate}
\item $f$ is said to be uniformly expansive if for all $A>0$ and $t>0$, there exist $B\in(0,1)$ and $t'>0$ such that $M_2(f(x),f(x'),t')\geq B$ whenever $M_1(x,x',t)\geq A$ for $x,x'\in X$.
\item $f$ is said to be effectively proper if for all $C>0$ and $t>0$, there exist $D\in(0,1)$ and $t'>0$ such that $M_1(x,x',t')\geq D$ whenever $M_2(f(x),f(x'),t)\geq C$ for $x,x'\in X$.
\item $f$ is a coarse embedding if it is both uniformly expansive and effectively proper.
\item $f$ is a coarse equivalence if it is a coarse embedding and it is coarsely onto in the sense that there exist $r\in(0,1)$ and $t>0$ such that for each $y\in Y$ there exists $x\in X$ satisfying $M_2(f(x),y,t)>1-r$.
\end{enumerate}
\end{defn}

\begin{defn}
Let $X$ be a set, and let $(Y,M,*)$ be a fuzzy metric space. Let $f:X\rightarrow Y$ and $g:X\rightarrow Y$ be functions. Then $f$ is close to $g$, denoted $f\sim g$, if there exist $r\in(0,1)$ and $t>0$ such that $M(f(x),g(x),t)>1-r$ for all $x\in X$.
\end{defn}

\begin{prop} \cite[Proposition 5.4]{Grz}
Let $(X,M_1,*_1)$ and $(Y,M_2,*_2)$ be fuzzy metric spaces. Let $f:X\rightarrow Y$ be a function. Then $f$ is a coarse equivalence if and only if $f$ is uniformly expansive and there exists a uniformly expansive $g:Y\rightarrow X$ such that the compositions $f\circ g$ and $g\circ f$ are close to the identity maps of $Y$ and $X$ respectively. 

The function $g$ is called a coarse inverse of $f$.
\end{prop}

\begin{thm} \label{invar}
Let $(X,M_1,*_1)$ and $(Y,M_2,*_2)$ be uniformly locally finite fuzzy metric spaces. Let $f:X\rightarrow Y$ be a coarse equivalence. Then $(X,M_1,*_1)$ has property A if and only if $(Y,M_2,*_2)$ does.
\end{thm}

\begin{proof}
Assume that $f:X\rightarrow Y$ is a coarse equivalence, and $g:Y\rightarrow X$ is a coarse inverse.
Assume that $(X,M_1,*_1)$ has property A, and let $\varepsilon>0$, $t>0$, $r\in(0,1)$ be given.
Since $g$ is uniformly expansive, there exist $B\in(0,1)$ and $t_1>0$ such that $M_1(g(y),g(y'),t_1)\geq B$ whenever $M_2(y,y',t)\geq 1-r$.
From the definition of property A, there exist $t_2>0$, $r_2>0$, and a family $\{A_x\}_{x\in X}$ of non-empty finite subsets of $X\times\mathbb{N}$ such that $A_x\subset B(x,r_2,t_2)\times\mathbb{N}$ for each $x\in X$, and $|A_x\Delta A_{x'}|<\varepsilon|A_x\cap A_{x'}|$ whenever $M_1(x,x',t)>B$.
Fix $y_0\in Y$. For each $y\in Y$, let $n_y=|(f^{-1}(y)\times\mathbb{N})\cap A_{g(y_0)}|$, and define
\[ B_{y_0}=\bigcup_{y\in Y,n_y\neq 0}\{(y,1),(y,2),\ldots,(y,n_y)\}. \]
The sets $(f^{-1}(y)\times\mathbb{N})\cap A_{g(y_0)}$ partition $A_{g(y_0)}$, so $|B_{y_0}|=|A_{g(y_0)}|<\infty$.
This produces a family $\{B_y\}_{y\in Y}$ of non-empty finite subsets of $Y\times\mathbb{N}$.

If $(y',n)\in B_y$, then there exists $(x,m)\in(f^{-1}(y')\times\mathbb{N})\cap A_{g(y)}$.
Now $M_1(g(y),x,t_2)>1-r_2$, so $M_2(f(g(y)),y',t_3)\geq C$ for some $t_3>0$ and $C\in(0,1)$.
Since $f\circ g$ is close to the identity map on $Y$, there exist $r_4\in(0,1)$ and $t_4>0$ such that $M_2(y,f(g(y)),t_4)>1-r_4$, so $M_2(y,y',t_3+t_4)\geq M_2(y,f(g(y)),t_4)*_2M_2(f(g(y)),y',t_3)\geq C*_2(1-r_4)>1-r_5$ for some $r_5\in(0,1)$.

If $(x,m)\in A_{g(y)}$ for some $m$, then $(f(x),1)\in B_y$, and it follows that $|A_{g(y)}\cap A_{g(y')}|\leq |B_y\cap B_{y'}|$ and $|A_{g(y)}\Delta A_{g(y')}|\geq|B_y\Delta B_{y'}|$ for $y,y'\in Y$.
Hence, $|B_y\Delta B_{y'}|<\varepsilon|B_y\cap B_{y'}|$ whenever $M_2(y,y',t)>1-r$.
Therefore $(Y,M_2,*_2)$ has property A.
\end{proof}

\begin{rem}
The same proof as above shows that if $g:Y\rightarrow X$ is a coarse embedding, and $(X,M_1,*_1)$ has property A, then so does $(Y,M_2,*_2)$. This is because $g:(Y,M_2,*_2)\rightarrow (g(Y),M_2|_{g(Y)\times g(Y)\times(0,\infty)},*_2)$ is a coarse equivalence and we may consider a coarse inverse $f:g(Y)\rightarrow Y$. Therefore, property A is inherited by fuzzy metric subspaces.
\end{rem}

\section{Characterizations of property A}

In this section, we provide characterizations of property A for uniformly locally finite fuzzy metric spaces in terms of positive definite kernels, maps into Hilbert spaces, and operators on Hilbert space. We begin with definitions of terms that will appear in the result.

\begin{defn}
Let $X$ be a set. A map $k:X\times X\rightarrow\mathbb{C}$ (resp. $\mathbb{R}$) is called a positive definite kernel if for all finite sequences $x_1,\ldots,x_n$ in $X$ and $\lambda_1,\ldots,\lambda_n\in\mathbb{C}$ (resp. $\mathbb{R}$), we have $\sum_{i,j=1}^n\lambda_i\overline{\lambda_j} k(x_i,x_j)\geq 0$.
\end{defn}

\begin{defn} 
Let $(X,M,*)$ be a uniformly locally finite fuzzy metric space, and write $B(\ell^2(X))$ for the set of all bounded linear operators on $\ell^2(X)$.
An operator $S\in B(\ell^2(X))$ has finite propagation if there exist $t>0$ and $r\in(0,1)$ such that $\langle S\delta_y,\delta_x \rangle=0$ whenever $M(x,y,t)<1-r$.
Here, $\delta_x$ refers to the standard basis vector in $\ell^2(X)$ given by
\[ \delta_x(z)=\begin{cases} 1\;\text{if}\;z=x, \\ 0\;\text{otherwise}. \end{cases} \]

The uniform Roe algebra $C^*_u(X)$ is the operator norm closure in $B(\ell^2(X))$ of the set of all bounded linear operators on $\ell^2(X)$ with finite propagation.
\end{defn}

Note that the composition $S_1\circ S_2$ has finite propagation whenever $S_1$ and $S_2$ do. Indeed, if $\langle S_1\delta_y,\delta_x \rangle=0$ whenever $M(x,y,t_1)<1-r_1$, and $\langle S_2\delta_y,\delta_x \rangle=0$ whenever $M(x,y,t_2)<1-r_2$, then $\langle S_1\circ S_2\delta_y,\delta_x \rangle=0$ whenever $M(x,y,t_1+t_2)<(1-r_1)*(1-r_2)$.

The following follows from \cite[Theorem 3.1]{Sako} on general uniformly locally finite coarse spaces but we shall rewrite the proof in terms of fuzzy metric spaces. The reader may refer to \cite[Theorem 1.11]{Willett} for the analogous statements for uniformly locally finite metric spaces.

\begin{thm} \label{charThm}
The following are equivalent for a uniformly locally finite fuzzy metric space $(X,M,*)$:

\begin{enumerate}
\item $X$ has property A.
\item For every $\varepsilon>0$, $r\in(0,1)$, and $t>0$, there exists a map $\eta:X\rightarrow\ell^1(X)$ such that
	\begin{itemize}
		\item $||\eta_x||_1=1$ for all $x\in X$,
		\item $||\eta_x-\eta_y||_1<\varepsilon$ if $M(x,y,t)>1-r$,
		\item there exist $R\in(0,1)$ and $T>0$ such that the support of $\eta_x$ satisfies $\mathrm{supp}(\eta_x)\subseteq B(x,R,T)$ for all $x\in X$.
	\end{itemize}
\item For every $\varepsilon>0$, $r\in(0,1)$, and $t>0$, there exists a map $\eta:X\rightarrow\ell^2(X)$ such that
	\begin{itemize}
		\item $||\eta_x||_2=1$ for all $x\in X$,
		\item $||\eta_x-\eta_y||_2<\varepsilon$ if $M(x,y,t)>1-r$,
		\item there exist $R\in(0,1)$ and $T>0$ such that the support of $\eta_x$ satisfies $\mathrm{supp}(\eta_x)\subseteq B(x,R,T)$ for all $x\in X$.
	\end{itemize}
\item For every $\varepsilon>0$, $r\in(0,1)$, and $t>0$, there exist a Hilbert space $H$ and a map $\eta:X\rightarrow H$ such that
	\begin{itemize}
		\item $||\eta_x||=1$ for all $x\in X$,
		\item $||\eta_x-\eta_y||<\varepsilon$ if $M(x,y,t)>1-r$,
		\item there exist $R\in(0,1)$ and $T>0$ such that $\langle\eta_x,\eta_y\rangle=0$ whenever $M(x,y,T)< 1-R$.
	\end{itemize}
\item For every $\varepsilon>0$, $r\in(0,1)$, and $t>0$, there exist $R\in(0,1)$, $T>0$, and a positive definite kernel $k:X\times X\rightarrow\mathbb{R}$ such that
	\begin{itemize}
		\item $|1-k(x,y)|<\varepsilon$ if $M(x,y,t)>1-r$,
		\item there exist $R\in(0,1)$ and $T>0$ such that $k(x,y)=0$ whenever $M(x,y,T)< 1-R$.
	\end{itemize}
\item For every $\varepsilon>0$, $r\in(0,1)$, and $t>0$, there exist $R\in(0,1)$, $T>0$, and a positive definite kernel $k:X\times X\rightarrow\mathbb{C}$ such that
	\begin{itemize}
		\item $|1-k(x,y)|<\varepsilon$ if $M(x,y,t)>1-r$,
		\item there exist $R\in(0,1)$ and $T>0$ such that $k(x,y)=0$ whenever $M(x,y,T)< 1-R$,
		\item convolution with $k$ defines a bounded linear operator $S_k$ belonging to the uniform Roe algebra $C^*_u(X)$.
	\end{itemize}
\end{enumerate}
\end{thm}

\begin{proof}
(i) $\Rightarrow$ (ii):
Suppose $X$ has property A. Fix $\varepsilon>0$, $r\in(0,1)$, and $t>0$.
There exist $r'\in(0,1)$, $t'>0$, and a family $\{A_x\}_{x\in X}$ of non-empty finite subsets of $X\times\mathbb{N}$ such that $A_x\subset B(x,r',t')\times\mathbb{N}$ for each $x\in X$, and $|A_x\Delta A_y|<\frac{\varepsilon}{2}|A_x\cap A_y|$ whenever $M(x,y,t)>1-r$.

Define $A_x(y)\subset\mathbb{N}$ by
\[ A_x(y)=\{n\in\mathbb{N}:(y,n)\in A_x\}. \]
For $x,y\in X$, define $\zeta_x(y)=|A_x(y)|$ and $\eta_x=\zeta_x/||\zeta_x||_1$. Then $\eta_x$ is a unit vector in $\ell^1(X)$ for each $x\in X$.

If $M(x,y,t)>1-r$, then
\begin{align*}
||\eta_x-\eta_y||_1 &= \frac{||\;||\zeta_y||_1\zeta_x - ||\zeta_x||_1\zeta_y\;||_1}{||\zeta_x||_1||\zeta_y||_1} \\
&\leq \frac{|\;||\zeta_y||_1-||\zeta_x||_1\;|\cdot||\zeta_x||_1 + ||\zeta_x||_1\cdot||\zeta_x-\zeta_y||_1}{||\zeta_x||_1||\zeta_y||_1} \\
&\leq 2\frac{||\zeta_x-\zeta_y||_1}{||\zeta_y||_1} \\
&\leq 2\frac{|A_x\Delta A_y|}{|A_y|} \\
&<\varepsilon.
\end{align*}
Moreover, $\eta_x(y)\neq 0$ if and only if $\zeta_x(y)\neq 0$, and in this case $(y,n)\in A_x$ for some $n\in\mathbb{N}$, so $y\in B(x,r',t')$.

(ii) $\Rightarrow$ (iii):
Assuming (ii), fix $\varepsilon>0$, $r\in(0,1)$, and $t>0$.
Then there is a map $\xi:X\rightarrow\ell^1(X)$ such that
	\begin{itemize}
		\item $||\xi_x||_1=1$ for all $x\in X$,
		\item $||\xi_x-\xi_y||_1<\varepsilon^2$ if $M(x,y,t)>1-r$,
		\item there exist $R\in(0,1)$ and $T>0$ such that the support of $\xi_x$ satisfies $\mathrm{supp}(\xi_x)\subseteq B(x,R,T)$ for all $x\in X$.
	\end{itemize}
Replacing $\xi_x$ by the map $y\mapsto|\xi_x(y)|$, we may assume $\xi_x$ takes nonnegative values while retaining the properties above.
For each $x\in X$, define a unit vector $\eta_x$ in $\ell^2(X)$ by $\eta_x(y)=\sqrt{\xi_x(y)}$.
If $M(x,y,t)>1-r$, then \[ ||\eta_x-\eta_y||_2\leq\sqrt{||\xi_x-\xi_y||_1}<\varepsilon. \]
Moreover, $\mathrm{supp}(\eta_x)=\mathrm{supp}(\xi_x)\subseteq B(x,R,T)$ for all $x\in X$.

(iii) $\Rightarrow$ (iv):
Assuming (iii), if $\langle\eta_x,\eta_y\rangle\neq 0$, then there exists $z\in\mathrm{supp}(\eta_x)\cap\mathrm{supp}(\eta_y)$, so $M(x,z,T)>1-R$ and $M(z,y,T)=M(y,z,T)>1-R$. Hence, $M(x,y,2T)\geq M(x,z,T)*M(z,y,T)\geq (1-R)*(1-R)=1-R'$ for some $R'\in(0,1)$.

(iv) $\Rightarrow$ (v):
Assuming (iv), we get a positive definite kernel $k:X\times X\rightarrow\mathbb{R}$ by $k(x,y)=Re(\langle\eta_x,\eta_y\rangle)$. 
Now $||\eta_x-\eta_y||^2=||\eta_x||^2+||\eta_y||^2-2Re(\langle\eta_x,\eta_y\rangle)=2-2Re(\langle\eta_x,\eta_y\rangle)$, so if $M(x,y,t)>1-r$, then $|1-k(x,y)|=|1-Re(\langle\eta_x,\eta_y\rangle)|=||\eta_x-\eta_y||^2/2<\varepsilon^2/2$.
Also, $k(x,y)=0$ whenever $M(x,y,T)<1-R$.

(v) $\Rightarrow$ (vi):
Taking $k,R,T$ from (v), since $X$ is uniformly locally finite, there exists $N_{R,T}$ such that $|B(x,R,T)|<N_{R,T}$ for all $x\in X$, so for each $x\in X$, we have $|\{y\in X:k(x,y)\neq 0\}|\leq N_{R,T}$.
Also, $|k(x,y)|\leq\sqrt{|k(x,x)k(y,y)|}=1$ by Cauchy-Schwarz.
Convolution with $k$ defines a linear map $S_k:\ell^2(X)\rightarrow\ell^2(X)$ given by \[ (S_k\xi)_x=\sum_{y\in X}k(x,y)\xi_y. \]
Now
\begin{align*}
||S_k\xi||_2^2 &= \sum_{x\in X}|\sum_{y\in B(x,R,T)}k(x,y)\xi_y|^2 \leq \sum_{x\in X}(\sum_{y\in B(x,R,T)}|\xi_y|)^2 \\
&\leq \sum_{y\in X}N_{R,T}\sum_{x\in B(y,R,T)}|\xi_y|^2\leq\sum_{y\in X}N_{R,T}^2|\xi_y|^2= N_{R,T}^2||\xi||_2^2.
\end{align*}
Hence, $S_k$ is a bounded operator on $\ell^2(X)$.
Moreover, $\langle S_k\delta_y,\delta_x \rangle=k(x,y)$, so $S_k$ has finite propagation.

(vi) $\Rightarrow$ (iii):
Assuming (vi), note that $S_k$ is a positive operator by seeing that $\langle S_kf,f \rangle\geq 0$ for all $f$ in the dense subset of all finitely supported functions in $\ell^2(X)$.
Thus, it has a unique positve square root $S_l$, and there is an operator $S_m$ with finite propagation such that
\[ ||S_l-S_m||<\min(\varepsilon,\frac{\varepsilon}{2(||S_l||+\varepsilon)}). \]
By passing to $(S_m+S_m^*)/2$, we may assume that $S_m$ is self-adjoint.
The operator $S_m$ corresponds to a kernel $m$, and there exist $R'\in(0,1)$ and $T'>0$ such that $m(x,y)=0$ whenever $M(x,y,T')<1-R'$.
Also,
\[ ||S_k-S_m^2||=||S_l^2-S_m^2||\leq ||S_l|| ||S_l-S_m||+||S_l-S_m|| ||S_m||<\varepsilon. \]
For each $x\in X$, define $\theta_x\in\ell^2(X)$ by $\theta_x(y)=m(x,y)$.
Note that \[ \langle\theta_x,\theta_y\rangle=\sum_{z\in X}m(x,z)\overline{m(y,z)}=\sum_{z\in X}m(x,z)m(z,y), \]
so $|\langle\theta_x,\theta_y\rangle-k(x,y)|<\varepsilon$.
If $M(x,y,t)>1-r$, then
\[ |1-\langle\theta_x,\theta_y\rangle|\leq |\langle\theta_x,\theta_y\rangle-k(x,y)|+|1-k(x,y)|<2\varepsilon, \]
so $||\theta_x-\theta_y||_2^2=\langle\theta_x,\theta_x\rangle+\langle\theta_y,\theta_y\rangle-2Re(\langle\theta_x,\theta_y\rangle)<8\varepsilon$. 
Also, $||\theta_x||_2^2>1-2\varepsilon$ so setting $\eta_x=\theta_x/||\theta_x||_2^2$, we have
\begin{align*}
||\eta_x-\eta_y||_2^2 &\leq \frac{|\;||\theta_y||_2-||\theta_x||_2\;|\cdot||\theta_x||_2 + ||\theta_x||_2\cdot||\theta_x-\theta_y||_2}{||\theta_x||_2||\theta_y||_2} \\
&\leq 2\frac{||\theta_x-\theta_y||_2}{||\theta_y||_2} \\
&< 2\sqrt{\frac{8\varepsilon}{1-2\varepsilon}}.
\end{align*}
Moreover, since $m(x,y)=0$ whenever $M(x,y,T')<1-R'$, we have $\mathrm{supp}(\eta_x)\subseteq B(x,R',T')$.

(iii) $\Rightarrow$ (i):
Assuming (iii), fix $\varepsilon>0$, $t>0$, and $r\in(0,1)$.
There exist unit vectors $\{\eta_x\}_{x\in X}$ in $\ell^2(X)$ such that $||\eta_x-\eta_y||<\varepsilon$ whenever $M(x,y,t)>1-r$, and there exist $R\in(0,1)$ and $T>0$ such that $\mathrm{supp}(\eta_x)\subseteq B(x,R,T)$ for all $x\in X$.
Since $X$ is uniformly locally finite, there exists $N_{R,T}$ such that $|\mathrm{supp}(\eta_x)|<N_{R,T}$ for all $x\in X$.

Define unit vectors $\xi_x\in\ell^1(X)$ by $\xi_x(y)=|\eta_x(y)|^2$. 
If $M(x,y,t)>1-r$, then $||\xi_x-\xi_y||_1\leq ||\eta_x+\eta_y||_2||\eta_x-\eta_y||_2<2\varepsilon$.
Choose $N\in\mathbb{N}$ such that $N>N_{R,T}/\varepsilon$, and define $\zeta_x(y)=j/N$, where $j\in\{0,\ldots,N\}$ and $j-1< N\xi_x(y)\leq j$.
Then $||\zeta_x-\xi_x||_1\leq N_{R,T}/N<\varepsilon$.
Thus, if $M(x,y,t)>1-r$, then $||\zeta_x-\zeta_y||_1<4\varepsilon$.

For $x\in X$, define $A_x\subset X\times\mathbb{N}$ by declaring that $(x',j)\in A_x$ if and only if $0<j\leq N\zeta_x(x')$.
Then $A_x\subset B(x,R,T)\times\mathbb{N}$ and $|A_x|=N||\zeta_x||_1$ for each $x\in X$.
Note that 
\begin{align*}
|A_{x_1}\Delta A_{x_2}| &=\sum_{x\in X}N|\zeta_{x_1}(x)-\zeta_{x_2}(x)|=N||\zeta_{x_1}-\zeta_{x_2}||_1 \\
&=\frac{|A_{x_1}|}{||\zeta_{x_1}||_1}||\zeta_{x_1}-\zeta_{x_2}||_1,
\end{align*}
so
\[ \frac{|A_{x_1}\Delta A_{x_2}|}{|A_{x_1}|}<\frac{||\zeta_{x_1}-\zeta_{x_2}||_1}{||\zeta_{x_1}||_1}<\frac{4\varepsilon}{1-\varepsilon}\] if $M(x_1,x_2,t)>1-r$.
In this case, \begin{align*} \frac{4\varepsilon}{1-\varepsilon}>\frac{|A_{x_1}\Delta A_{x_2}|}{|A_{x_1}|}&=\frac{|A_{x_1}|+|A_{x_2}|-2|A_{x_1}\cap A_{x_2}|}{|A_{x_1}|} \\ &>\frac{2(1-\varepsilon)}{1+\varepsilon}-\frac{2|A_{x_1}\cap A_{x_2}|}{N(1+\varepsilon)}, \end{align*}
so 
\begin{align*}
|A_{x_1}\cap A_{x_2}| > N(1-\varepsilon)-\frac{2N\varepsilon(1+\varepsilon)}{1-\varepsilon} = \frac{1-4\varepsilon-\varepsilon^2}{1-\varepsilon}N.
\end{align*}
Hence,
\begin{align*}
\frac{|A_{x_1}\Delta A_{x_2}|}{|A_{x_1}\cap A_{x_2}|} &< \frac{2N(1+\varepsilon)(1-\varepsilon)-2N(1-4\varepsilon-\varepsilon^2)}{N(1-4\varepsilon-\varepsilon^2)} \\
&= \frac{8\varepsilon}{1-4\varepsilon-\varepsilon^2},
\end{align*}
which tends to zero as $\varepsilon$ tends to 0, showing that $X$ has property A.
\end{proof}

\begin{rem}
In the proof above, the assumption of uniform local finiteness was used only in the implications (v) $\Rightarrow$ (vi) and (iii) $\Rightarrow$ (i).
\end{rem}

Using (ii) in the previous theorem, we will refine Theorem \ref{thm:asdim} to show that if the multiplicities of certain uniformly bounded covers of a fuzzy metric space grow at a subexponential rate, then the space has property A.

For the next result, we will write $\mathrm{mult}(\mathcal{U})$ for the multiplicity of a cover $\mathcal{U}$, and we will write $L(\mathcal{U})\geq(r,t)$ to mean that $\mathcal{U}$ has a Lebesgue pair with values at least $r$ and $t$.
Using Lemma \ref{lem:balls} and \cite[Proposition 3.4]{Grz}, one recovers \cite[Theorem 1]{Oz} when our result is applied to the standard fuzzy metric space corresponding to a metric space.

We shall say that a function $f:(0,\infty)\rightarrow\mathbb{R}$ has subexponential growth if $\lim_{t\rightarrow\infty}\frac{\ln f(t)}{t}=0$.

\begin{thm}
Let $(X,M,*)$ be a uniformly locally finite fuzzy metric space.
Suppose there exists $r\in(0,1)$ such that the function
\[ ad_X(t)=\min\{ \mathrm{mult}(\mathcal{U}):\mathcal{U}\;\text{uniformly bounded cover of}\;X,L(\mathcal{U})\geq(r',t) \}-1 \]
has subexponential growth, where $r'=1-\frac{1}{2}(1-r)^{*(2)}$. % just need $r'>1-(1-r)^{*(2)}$
Then $(X,M,*)$ has property A.
\end{thm}

\begin{proof}
Given $\varepsilon>0$, $r\in(0,1)$, and $t>0$, we will construct a map $\eta:X\rightarrow\ell^1(X)$ such that
\begin{itemize}
\item $||\eta_x||_1=1$ for all $x\in X$,
\item $||\eta_x-\eta_y||_1<\varepsilon$ if $M(x,y,t)>1-r$,
\item there exist $R\in(0,1)$ and $T>0$ such that $\mathrm{supp}(\eta_x)\subseteq B(x,R,T)$ for all $x\in X$.
\end{itemize}
Let $\mathcal{U}=\{U_i\}_{i\in I}$ be a uniformly bounded cover of $X$ with Lebesgue pair at least $(r',6n)$ and multiplicity $ad_X(6n)+1$.
Uniform boundedness gives $R\in(0,1)$ and $T>0$ such that $M(x,y,T)>1-R$ for all $x,y\in U_i$ and $i\in I$.

For each $i\in I$, choose $x_i\in U_i$ and let $J:\ell^1(I)\rightarrow\ell^1(X)$ be the contraction sending $\delta_i$ to $\delta_{x_i}$.
For each $x\in X$ and $k\in\mathbb{N}$, set \[S_x(r,k)=\{i\in I:B(x,r,k)\subset U_i\}.\]
For $x,y\in X$ and $t\in(0,k)$, if $M(x,y,t)>1-r$ and $M(y,z,k)>1-r$, then $M(x,z,k+t)\geq(1-r)*(1-r)>1-r'$. Also, if $M(x,y,t)>1-r$ and $M(x,z,k-t)>1-r$, then $M(y,z,k)\geq M(x,y,t)*M(x,z,k-t)\geq(1-r)*(1-r)>1-r'$,
so we have
\[ S_x(r',k+t)\subseteq S_x(r,k)\cap S_y(r,k)\subseteq S_x(r,k)\cup S_y(r,k)\subseteq S_x(r',k-t). \]
For each non-empty finite set $S$, set $\xi_S=|S|^{-1}\chi_S$, where $\chi_S$ denotes the characteristic function of $S$.
Then for any non-empty finite sets $S$ and $T$, \[ ||\xi_S-\xi_T||_1=2\left( 1-\frac{|S\cap T|}{\max(|S|,|T|)} \right). \]
Define \[ \zeta_x^n=\frac{1}{n}\sum_{k=n+1}^{2n}\xi_{S_x(r,k)}\in\ell^1(I) \quad\text{and}\quad\eta_x^n=J(\zeta_x^n)\in\ell^1(X). \]
Note that $||\eta_x^n||_1=1$ for all $x\in X$.
If $x_i\in\mathrm{supp}(\eta_x^n)$, then $i\in S_x(r,k)$ for some $k$.
Thus, $\mathrm{supp}(\eta_x^n)\subseteq B(x,R,T)$ for all $x\in X$.

Now assume $M(x,y,t)>1-r$ with $t\in\{1,\ldots,k-1\}$, and $n\geq t$.
Then
\[ ||\zeta_x^n-\zeta_y^n||_1\leq\frac{1}{n}\sum_{k=n+1}^{2n}||\xi_{S_x(r,k)}-\xi_{S_y(r,k)}||_1\leq\frac{2}{n}\sum_{k=n+1}^{2n}\left( 1-\frac{|S_x(r',k+t)|}{|S_x(r',k-t)|} \right). \]
On the other hand,
\begin{align*} \frac{1}{n}\sum_{k=n+1}^{2n}\frac{|S_x(r',k+t)|}{|S_x(r',k-t)|} &\geq \left( \prod_{k=n+1}^{2n}\frac{|S_x(r',k+t)|}{|S_x(r',k-t)|} \right)^{1/n} \\
&= \left( \frac{\prod_{k=2n-t+1}^{2n+t}|S_x(r',k)|}{\prod_{k=n-t+1}^{n+t}|S_x(r',k)|} \right)^{1/n} \\
&\geq \mathrm{mult}(\mathcal{U})^{-2t/n} \end{align*}
since $1\leq |S_x(r',k)|\leq\mathrm{mult}(\mathcal{U})$ for each $k\in\{n-t+1,\ldots,2n+t\}$.
Hence,
\[ ||\eta_x^n-\eta_y^n||_1 \leq ||\zeta_x^n-\zeta_y^n||_1 \leq 2(1-\mathrm{mult}(\mathcal{U})^{-2t/n}). \]
Since $\mathrm{mult}(\mathcal{U})=ad_X(6n)+1$ grows subexponentially in $n$, one sees that $||\eta_x^n-\eta_y^n||_1<\varepsilon$ for all sufficiently large $n$.
\end{proof}

% relation between $ad_X(t)$ and asdim(X): $ad_X(t)$ non-decreasing in $t$, bdd above by asdim(X), not clear that $\lim_{t\rightarrow\infty}ad_X(t)=asdim(X)$ because of $r$

\section{Coarse embeddability into Hilbert space} 

The following lemma shows that coarse embeddability of a fuzzy metric space into the standard fuzzy metric space corresponding to a metric can be verified using the metric instead of the standard fuzzy metric.
Using it, we will show that uniformly locally finite fuzzy metric spaces with property A are coarsely embeddable into Hilbert space, which is in line with the original motivation for the introduction of property A for metric spaces as shown in \cite[Theorem 2.2]{Yu00}.

We shall omit the proof of the lemma as it is straightforward using the definitions and Lemma \ref{lem:balls}.

\begin{lem}
Let $(X,M_1,*_1)$ be a fuzzy metric space, let $(Y,d)$ be a metric space, and let $(Y,M_2,*_2)$ be the corresponding standard fuzzy metric space.
\begin{enumerate}
\item A map $f:(X,M_1,*_1)\rightarrow (Y,M_2,*_2)$ is uniformly expansive if and only if for all $A>0$ and $t>0$, there exists $S>0$ such that $d(f(x),f(x'))\leq S$ whenever $M_1(x,x',t)\geq A$.
\item A map $f:(X,M_1,*_1)\rightarrow (Y,M_2,*_2)$ is effectively proper if and only if for all $R>0$ there exists $D\in(0,1)$ and $t>0$ such that $M_1(x,x',t)\geq D$ whenever $d(f(x),f(x'))\leq R$.
\end{enumerate}
\end{lem}

\begin{rem}
One can reverse the roles of $X$ and $Y$ above to get analogous statements when $X$ is a metric space instead.
Consequently, one sees that there is a coarse embedding between two metric spaces if and only if there is a coarse embedding between the corresponding standard fuzzy metric spaces, and similarly for coarse equivalences.
\end{rem}

\begin{lem} \label{lem:Ace}
Let $(X,M,*)$ be a uniformly locally finite fuzzy metric space with property A. For every $\varepsilon>0$, $r\in(0,1)$, and $t>0$, there exists a map $\xi:X\rightarrow\ell^2(X\times\mathbb{N})$ such that
\begin{itemize}
\item $||\xi_x||_2=1$ for all $x\in X$,
\item $||\xi_x-\xi_y||_2<\varepsilon$ if $M(x,y,t)>1-r$,
\item there exist $R\in(0,1)$ and $T>0$ such that the support of $\xi_x$ satisfies $\mathrm{supp}(\xi_x)\subseteq B(x,R,T)\times\mathbb{N}$ for all $x\in X$.
\end{itemize}
\end{lem}

\begin{proof}
Given $\varepsilon>0$, $r\in(0,1)$, and $t>0$, there exist $R\in(0,1)$, $T>0$, and a family $\{A_x\}_{x\in X}$ of non-empty finite subsets of $X\times\mathbb{N}$ such that $A_x\subset B(x,R,T)\times\mathbb{N}$ for each $x\in X$, and $|A_x\Delta A_y|<\varepsilon^2|A_x\cap A_y|$ whenever $M(x,y,t)>1-r$.

Write $\chi_{A_x}$ for the characteristic function of $A_x$, and set $\xi_x=|A_x|^{-1/2}\chi_{A_x}$ for each $x\in X$.
If $M(x,y,t)>1-r$, then \[ |A_x|+|A_y|=2|A_x\cap A_y|+|A_x\Delta A_y|< (2+\varepsilon^2)|A_x\cap A_y|, \]
so \[ \langle\xi_x,\xi_y\rangle=\frac{|A_x\cap A_y|}{\sqrt{|A_x| |A_y|}}\geq\frac{2|A_x\cap A_y|}{|A_x|+|A_y|}>\frac{2}{2+\varepsilon^2}, \]
and $||\xi_x-\xi_y||_2^2=\langle \xi_x-\xi_y,\xi_x-\xi_y \rangle=2(1-\langle \xi_x,\xi_y \rangle)<\frac{2\varepsilon^2}{2+\varepsilon^2}<\varepsilon^2$.
\end{proof}

\begin{thm} \label{CE}
A uniformly locally finite fuzzy metric space with property A coarsely embeds into a (separable) Hilbert space.
\end{thm}

\begin{proof}
For each $n\in\mathbb{N}$, choose $\xi^n$ as in Lemma \ref{lem:Ace} corresponding to $\varepsilon=2^{-n}$, $t=n$, and $r_n\in(0,1)$ with $r_n$ increasing to 1.
Let $R_n\in(0,1)$ and $T_n>0$ be such that $\mathrm{supp}(\xi_x^n)\subseteq B(x,R_n,T_n)\times\mathbb{N}$.
Then $||\xi_x^n-\xi_y^n||_2=\sqrt{2}$ when $M(x,y,2T_n)<(1-R_n)*(1-R_n)$. 
%By increasing $T_n$ and $R_n$ if necessary, we may also assume that $T_n\geq n/2$ and $R_n\geq r_n$ for each $n$.
%\textbf{want opposite ineq instead}

Fix $z\in X$, and define a map $F:X\rightarrow\bigoplus_{n=1}^\infty\ell^2(X\times\mathbb{N})$ by
\[ F(x)=\bigoplus_{n=1}^\infty(\xi_x^n-\xi_z^n). \]
Suppose $x\neq y$, and let $M=\sup_nM(x,y,n)$.
There exists $N$ such that $r_n>1-\frac{M}{2}$ for all $n> N$.
There exists $N'$ such that $M(x,y,n)>\frac{M}{2}$ for all $n> N'$.
Hence, $1-r_n<\frac{M}{2}<M(x,y,n)$ for all $n> N''=\max(N,N')$, so $||\xi_x^n-\xi_y^n||_2<2^{-n}$ for all $n> N''$.
It follows that
\begin{align*}
||F(x)-F(y)||_2^2 &= \sum_{n=1}^\infty||\xi_x^n-\xi_y^n||_2^2 \\
&\leq \sum_{n=1}^{N''}||\xi_x^n-\xi_y^n||_2^2 + \sum_{n=N''+1}^\infty 2^{-n} \\
&< 4N''+1. 
\end{align*}
In the case where $y=z$, the calculation above shows that $F(x)$ indeed belongs to $\bigoplus_{n=1}^\infty\ell^2(X\times\mathbb{N})$.

Given $R>0$, if $||F(x)-F(y)||\leq R$, then the set $\{ n\in\mathbb{N}:M(x,y,2T_n)<(1-R_n)*(1-R_n) \}$ has at most $R^2/2$ elements. Letting $N=\max\{ n\in\mathbb{N}:M(x,y,2T_n)<(1-R_n)*(1-R_n) \}$, we have $M(x,y,2T_{N+1})\geq(1-R_{N+1})*(1-R_{N+1})$, so $F$ is effectively proper.

On the other hand, let $A>0$ and $t>0$. If $A\geq 1$ and $M(x,y,t)\geq A$, then $x=y$, so assume $A\in(0,1)$. 
We may then choose $n\geq t$ such that $r_n\geq 1-A$. If $M(x,y,t)\geq A$, then $M(x,y,n)\geq M(x,y,t)\geq A\geq 1-r_n$, so $||F(x)-F(y)||\leq 2^{-n}$, and $F$ is uniformly expansive.
\end{proof}

% in proof, for smaller $t$ larger$ A$, have smaller $n$ and larger norm diff; for smaller R, have smaller N

Examples of uniformly locally finite metric spaces without property A and are coarsely embeddable into Hilbert space were constructed in \cite{AGS}, so the standard fuzzy metric spaces corresponding to these metric spaces also do not have property A but are coarsely embeddable into Hilbert space.

\appendix
\section{Coarse Spaces}

In this appendix, we collect statements about general coarse spaces and relate them to the context of fuzzy metric spaces in connection with the content in the main body of this paper.

Given a set $X$, and subsets $E,E_1,E_2\subseteq X\times X$, we define
\begin{align*}
E^{-1} &= \{(y,x)\in X\times X:(x,y)\in E\}, \\
E_1\circ E_2 &= \{(x,y)\in X\times X: \exists z\in X, (x,z)\in E_1,(z,y)\in E_2\}.
\end{align*}

\begin{defn} \cite[Definition 2.3]{RoeL}
A coarse structure on a set $X$ is a collection $\mathcal{E}$ of subsets of $X\times X$ satisfying
\begin{itemize}
\item $\{(x,x):x\in X\}\in\mathcal{E}$,
\item If $E\in\mathcal{E}$, then $E^{-1}\in\mathcal{E}$,
\item If $E_1,E_2\in\mathcal{E}$, then $E_1\circ E_2\in\mathcal{E}$,
\item If $E_1,E_2\in\mathcal{E}$, then $E_1\cup E_2\in\mathcal{E}$,
\item If $E_1\in\mathcal{E}$ and $E_2\subseteq E_1$, then $E_2\in\mathcal{E}$.
\end{itemize}
The pair $(X,\mathcal{E})$ is called a coarse space.
\end{defn}

Given a fuzzy metric space $(X,M,*)$, assuming that $a*b\neq 0$ whenever $a\neq 0$ and $b\neq 0$, the collection of all subsets $E\subseteq X\times X$ for which there exist $r\in(0,1)$ and $t>0$ such that $M(x,y,t)>1-r$ for all $(x,y)\in E$ forms a coarse structure on $X$. We will call this the bounded coarse structure and denote it by $\mathcal{E}_b$.

\begin{defn} \cite[Definitions 9.1 and 9.4]{RoeL}
Let $(X,\mathcal{E})$ be a coarse space. 
\begin{enumerate}
\item A family $\{D_i\}_{i\in I}$ of subsets of $X$ is said to be uniformly bounded if $\bigcup_{i\in I}D_i\times D_i\in\mathcal{E}$.
\item Given $E\in\mathcal{E}$, a family $\{D_i\}_{i\in I}$ of subsets of $X$ is $E$-disjoint if $(D_i\times D_j)\cap E=\emptyset$ whenever $i\neq j$.
\item $(X,\mathcal{E})$ has asymptotic dimension at most $n$ if for every $E\in\mathcal{E}$, $X$ can be covered by at most $n+1$ $E$-disjoint families $X_0,\ldots,X_n$ of uniformly bounded subsets of $X$. 
\end{enumerate}
\end{defn}

When a fuzzy metric space is equipped with its bounded coarse structure, one sees that the definition above agrees with the definition of asymptotic dimension for fuzzy metric spaces in \cite[Definition 3.8]{Grz}.

\begin{prop}
A fuzzy metric space $(X,M,*)$ has asymptotic dimension (at most) $n$ if and only if the coarse space $(X,\mathcal{E}_b)$ has asymptotic dimension (at most) $n$.
\end{prop}

A coarse space $(X,\mathcal{E})$ is said to be uniformly locally finite if for each $E\in\mathcal{E}$, we have $\sup_{x\in X}|\{ y\in X:(y,x)\in E \}|<\infty$. 

When a fuzzy metric space is equipped with its bounded coarse structure, this notion of uniform local finiteness agrees with that in Definition \ref{ulfFM}.

\begin{defn} \cite[Definition 2.7]{Sako} \label{Acoarsesp}
A uniformly locally finite coarse space $(X,\mathcal{E})$ has property A if for every $\varepsilon>0$ and every $E\in\mathcal{E}$, there exist $F\in\mathcal{E}$ and $A\subset F\times\mathbb{N}$ such that
\begin{enumerate}
\item the set $A_x=\{(y,n)\in X\times\mathbb{N}:(x,y,n)\in A\}$ is finite and non-empty for each $x\in X$,
\item $|A_x\Delta A_y|<\varepsilon|A_x\cap A_y|$ for each $(x,y)\in E$.
\end{enumerate}
\end{defn}

In \cite[Theorem 3.1]{Sako}, this definition was shown to be equivalent to the one given by Roe in \cite[Definition 11.35]{RoeL}.
We also note that the definition still makes sense without uniform local finiteness, and this was considered in the doctoral thesis of Bunn \cite[Definition 5.5]{Bunn}.

\begin{prop} \label{Aequiv}
A uniformly locally finite fuzzy metric space $(X,M,*)$ has property A if and only if the coarse space $(X,\mathcal{E}_b)$ has property A.
\end{prop}

\begin{proof}
Suppose $(X,M,*)$ has property A.
Given $\varepsilon>0$ and $E\in\mathcal{E}_b$, there exist $r\in(0,1)$ and $t>0$ such that $M(x,y,t)>1-r$ for all $(x,y)\in E$, so $|A_x\Delta A_y|<\varepsilon|A_x\cap A_y|$ for each $(x,y)\in E$.
There also exist $r'\in(0,1)$, $t>0$, and a family $\{A_x\}_{x\in X}$ of non-empty finite subsets of $X\times\mathbb{N}$ such that $A_x\subset B(x,r',t')\times\mathbb{N}$ for each $x\in X$.

Define $F=\{(x,y)\in X\times X:\exists n\in\mathbb{N},(y,n)\in A_x\}$. If $(x,y)\in F$, then $M(x,y,t')>1-r'$, so $F\in\mathcal{E}_b$.
Finally, define $A\subset F\times\mathbb{N}$ by the condition that $(x,y,n)\in A$ if and only if $(y,n)\in A_x$. This shows that $(X,\mathcal{E}_b)$ has property A.

Conversely, suppose $(X,\mathcal{E}_b)$ has property A.
Given $\varepsilon>0$, $r\in(0,1)$, and $t>0$, we have $B(x,r,t)\in\mathcal{E}_b$, so there exist $F\in\mathcal{E}_b$ and $A\subset F\times\mathbb{N}$ satisfying (i) and (ii) in Definition \ref{Acoarsesp}.
Now there exist $r'\in(0,1)$ and $t'>0$ such that $M(x,y,t')>1-r'$ for all $(x,y)\in F$, so $A_x$ in Definition \ref{Acoarsesp}(i) is contained in $B(x,r',t')\times\mathbb{N}$, and $|A_x\Delta A_y|<\varepsilon|A_x\cap A_y|$ for each $(x,y)\in B(x,r,t)$. This shows that $(X,M,*)$ has property A.
\end{proof}

Given a coarse space $(X,\mathcal{E})$, a subset $B$ of $X$ is said to be bounded if $B\times B\in\mathcal{E}$.
Given a coarse space $(X,\mathcal{E})$ and a set $S$, two maps $f,f':S\rightarrow X$ are said to be close if $\{(f(s),f'(s)):s\in S\}\in\mathcal{E}$.

\begin{defn} \cite[Definition 2.21]{RoeL}
Let $(X,\mathcal{E})$ and $(Y,\mathcal{F})$ be coarse spaces, and let $f:X\rightarrow Y$ be a map.
\begin{enumerate}
\item The map $f$ is proper if $f^{-1}(B)$ is bounded for each bounded subset $B$ of $Y$.
\item The map $f$ is bornologous if $(f\times f)(E)\in\mathcal{F}$ for each $E\in\mathcal{E}$.
\item The map $f$ is a coarse map if it is both proper and bornologous.
\item The spaces $(X,\mathcal{E})$ and $(Y,\mathcal{F})$ are coarsely equivalent if there exist coarse maps $f:X\rightarrow Y$ and $g:Y\rightarrow X$ such that $f\circ g$ and $g\circ f$ are close to the identity maps on $X$ and $Y$ respectively.
\end{enumerate}
\end{defn}

In the context of fuzzy metric spaces, proper and bornologous correspond respectively to effectively proper and uniformly expansive in Definition \ref{ceFM}.
The next proposition corresponds to Theorem \ref{invar} when applied to uniformly locally finite fuzzy metric spaces with their bounded coarse structures.
This result should be known to experts but we could not find a reference for it.

\begin{prop}
Let $(X,\mathcal{E})$ and $(Y,\mathcal{F})$ be uniformly locally finite coarse spaces. Let $f:X\rightarrow Y$ be a coarse equivalence.
Then $(X,\mathcal{E})$ has property A if and only if $(Y,\mathcal{F})$ does.
\end{prop}

\begin{proof}
Assume that $f:X\rightarrow Y$ is a coarse equivalence, and $g:Y\rightarrow X$ is a coarse map such that $f\circ g$ and $g\circ f$ are close to the identity maps on $X$ and $Y$ respectively.
Assume that $(X,\mathcal{E})$ has property A, and let $\varepsilon>0$ and $F\in\mathcal{F}$ be given.

Since $g$ is bornologous, we have $(g\times g)(F)\in\mathcal{E}$.
From the definition of property A, there exist $E\in\mathcal{E}$ and $A\subset E\times\mathbb{N}$ such that
\begin{enumerate}
\item the set $A_x=\{(x',n)\in X\times\mathbb{N}:(x,x',n)\in A\}$ is finite and non-empty for each $x\in X$,
\item $|A_x\Delta A_{x'}|<\varepsilon|A_x\cap A_{x'}|$ for each $(x,x')\in (g\times g)(F)$.
\end{enumerate}

Fix $y_0\in Y$. For each $y\in Y$, let $n_y=|(f^{-1}(y)\times\mathbb{N})\cap A_{g(y_0)}|$, and define
\[ B_{y_0}=\bigcup_{y\in Y,n_y\neq 0}\{(y,1),(y,2),\ldots,(y,n_y)\}.\]
The sets $(f^{-1}(y)\times\mathbb{N})\cap A_{g(y_0)}$ partition $A_{g(y_0)}$, so $|B_{y_0}|=|A_{g(y_0)}|<\infty$.
This produces a family $\{B_y\}_{y\in Y}$ of non-empty finite subsets of $Y\times\mathbb{N}$.

Define $B$ by the condition that $(y,y',n)\in B$ if and only if $(y',n)\in B_y$.
If $(y',n)\in B_y$, then there exists $(x,m)\in(f^{-1}(y')\times\mathbb{N})\cap A_{g(y)}$.
In particular, $(g(y),x,m)\in A\subset E\times\mathbb{N}$.
Since $f$ is bornologous, $(f(g(y)),y')$ belongs to $F_1=(f\times f)(E)\in\mathcal{F}$.
Since $f\circ g$ is close to the identity map on $Y$, we have $F_2=\{(y,f(g(y))):y\in Y\}\in\mathcal{F}$.
Hence, $(y,y')$ belongs to $F_2\circ F_1\in\mathcal{F}$, and $B\subset (F_2\circ F_1)\times\mathbb{N}$.

If $(x,m)\in A_{g(y)}$ for some $m$, then $(f(x),1)\in B_y$, and it follows that $|A_{g(y)}\cap A_{g(y')}|\leq |B_y\cap B_{y'}|$ and $|A_{g(y)}\Delta A_{g(y')}|\geq|B_y\Delta B_{y'}|$ for $y,y'\in Y$.
Hence, $|B_y\Delta B_{y'}|<\varepsilon|B_y\cap B_{y'}|$ whenever $(y,y')\in F$.
Therefore $(Y,\mathcal{F})$ has property A.
\end{proof}

\begin{thm} \cite[Remark 11.36(ii)]{RoeL}
If a uniformly locally finite coarse space has finite asymptotic dimension, then it has property A.
\end{thm}

This result still holds if one omits the condition of uniform local finiteness from the definition of property A \cite[Theorem 6.4]{Bunn}.

%\begin{lem} \label{lem:Ace1}
%Let $(X,\mathcal{E})$ be a uniformly locally finite coarse space with property A. For every $\varepsilon>0$ and $E\in\mathcal{E}$, there exists a map $\xi:X\rightarrow\ell^2(X\times\mathbb{N})$ such that
%\begin{itemize}
%\item $||\xi_x||_2=1$ for all $x\in X$,
%\item $||\xi_x-\xi_y||_2<\varepsilon$ if $(x,y)\in E$,
%\item there exists $F\in\mathcal{E}$ such that the support of $\xi_x$ satisfies $\{x\}\times\mathrm{supp}(\xi_x)\subseteq F\times\mathbb{N}$ for all $x\in X$.
%\end{itemize}
%\end{lem}
%
%\begin{proof}
%Given $\varepsilon>0$ and $E\in\mathcal{E}$, there exist $F\in\mathcal{E}$ and $A\subset F\times\mathbb{N}$ such that $A_x=\{(y,n)\in X\times\mathbb{N}:(x,y,n)\in A\}$ is finite and non-empty for each $x\in X$, and $|A_x\Delta A_y|<\varepsilon^2|A_x\cap A_y|$ whenever $(x,y)\in E$.
%
%Write $\chi_{A_x}$ for the characteristic function of $A_x$, and set $\xi_x=|A_x|^{-1/2}\chi_{A_x}$ for each $x\in X$.
%If $(x,y)\in E$, then \[ |A_x|+|A_y|=2|A_x\cap A_y|+|A_x\Delta A_y|< (2+\varepsilon^2)|A_x\cap A_y|, \]
%so \[ \langle\xi_x,\xi_y\rangle=\frac{|A_x\cap A_y|}{\sqrt{|A_x| |A_y|}}\geq\frac{2|A_x\cap A_y|}{|A_x|+|A_y|}>\frac{2}{2+\varepsilon^2}, \]
%and $||\xi_x-\xi_y||_2^2=\langle \xi_x-\xi_y,\xi_x-\xi_y \rangle=2(1-\langle \xi_x,\xi_y \rangle)<\frac{2\varepsilon^2}{2+\varepsilon^2}<\varepsilon^2$.
%\end{proof}

\begin{thm} \cite[Theorem 11.16(c) and Lemma 11.37]{RoeL}
If a uniformly locally finite coarse space has property A, then it is coarsely embeddable into a Hilbert space.
\end{thm}

When applied to a uniformly locally finite fuzzy metric space with its bounded coarse structure, these correspond to Theorems  \ref{thm:asdim} and \ref{CE}.

\bibliographystyle{plain}
\bibliography{mybib}

\end{document}